\documentclass[10pt]{amsart}
\usepackage{geometry}                
\usepackage[parfill]{parskip}    % Activate to begin paragraphs with an empty line rather than an indent
\usepackage{amssymb, amsthm, amsmath}
\usepackage{graphicx}
\usepackage{xypic}
\usepackage{subfig}
\usepackage{sidecap}
\usepackage{color}
\usepackage{dbnsymb}
\usepackage{hyperref}

\makeatletter
\@namedef{subjclassname@1991}{2020 Mathematics Subject Classification}
\makeatother

\makeatletter
\newtheorem*{rep@theorem}{\rep@title}
\newcommand{\newreptheorem}[2]{%
\newenvironment{rep#1}[1]{%
 \def\rep@title{#2 \ref{##1}}%
 \begin{rep@theorem}}%
 {\end{rep@theorem}}}
\makeatother

\newtheorem{theorem}{Theorem}[section]
\newtheorem{lemma}[theorem]{Lemma}
\newtheorem{proposition}[theorem]{Proposition}
\newtheorem{property}[theorem]{Property}
\newtheorem{corollary}[theorem]{Corollary}
\newtheorem{question}[theorem]{Question}

\theoremstyle{definition}
\newtheorem{definition}[theorem]{Definition}

\newtheorem{example}[theorem]{Example}

\newcommand{\os}{Ozsv\'ath and Szab\'o}
\newcommand{\spinc}{\text{Spin}^c}

\newcommand{\relspinc}{\underline{\mathrm{Spin}^c}}

\newcommand{\mft}{\mathfrak{t}}
\newcommand{\mfv}{\mathfrak{v}}
\newcommand{\mfh}{\mathfrak{h}}
\newcommand{\Z}{\mathbb{Z}}
\newcommand{\F}{\mathbb{F}}
\newcommand{\X}{\mathbb{X}}
\newcommand{\T}{\mathcal{T}}

\newcommand{\CFhat}{\widehat{CF}}
\newcommand{\Xhat}{\widehat{\X}}
\newcommand{\Ahat}{\widehat{A}}
\newcommand{\Bhat}{\widehat{B}}
\newcommand{\vhat}{\widehat{v}}
\newcommand{\hhat}{\widehat{h}}
\newcommand{\Fhat}{\widehat{F}}
\newcommand{\Aplus}{A^+}
\newcommand{\Bplus}{B^+}
\newcommand{\vplus}{v^+}
\newcommand{\hplus}{h^+}

\title[Band surgery and the signature of a knot]{A note on band surgery and the signature of a knot}
\author{Allison H. Moore}
\author{Mariel Vazquez}

\address{Allison H. Moore\\ Department of Mathematics \& Applied Mathematics \\ 
   Virginia Commonwealth University \\Richmond, VA 23284 \\ USA  
   }
\address{Mariel Vazquez\\  Department of Mathematics, Department of Microbiology and Molecular Genetics \\
   University of California, Davis \\ Davis, CA 95616 \\ USA   
   }


\begin{document}

\maketitle

\begin{abstract} 
Band surgery is an operation relating pairs of knots or links in the three-sphere. We prove that if two quasi-alternating knots $K$ and $K'$ of the same square-free determinant are related by a band surgery, then the absolute value of the difference in their signatures is either 0 or 8. This obstruction follows from a more general theorem about the difference in the Heegaard Floer $d$-invariants for pairs of L-spaces that are related by distance one Dehn fillings and satisfy a certain condition in first homology. These results imply that $T(2, 5)$ is the only torus knot $T(2, m)$ with $m$ square-free that admits a chirally cosmetic banding, i.e. a band surgery operation to its mirror image. We conclude with a discussion on the scarcity of chirally cosmetic bandings.
\end{abstract}

\section{Introduction}
Band surgery is an operation on knots or links in the three-sphere. Let $L$ be a link, and let $b:I \times I \rightarrow S^3$ be an embedding of the unit square with $L\cap b(I\times I) = b(I\times \partial I)$. Let $L'$ denote the link obtained by replacing $b(I\times \partial I)$ in $L$ with $b(\partial I\times I)$. We say that the link $L'$ results from \emph{band surgery} along $L$. See Figure \ref{fig:567} for two examples of band surgeries relating pairs of knots. When $L$ and $L'$ are oriented links and band surgery respects their orientations, the operation is called \emph{coherent} band surgery. Otherwise, it is called \emph{non-coherent}. Coherent band surgery converts a knot into a two-component link, whereas non-coherent band surgery converts a knot to another knot.

In this article we are concerned only with non-coherent band surgery operations relating unoriented knots. Elsewhere in the literature non-coherent band surgery is called \emph{incoherent} band surgery or \emph{unoriented banding} (e.g. \cite{AK, Kanenobu, IJ}). Band surgery is of interest in low-dimensional topology, in particular in knot theory, and in the study of surfaces in four-manifolds. It is of independent interest in DNA topology; for a more detailed commentary on this perspective, we refer the reader to \cite[Section 5]{LMV} and \cite{MV}, which review the relevance of band surgery to the study of DNA.

In Theorem \ref{thm:sigdif}, we find a new obstruction to band surgery for quasi-alternating knots via two knot invariants, the determinant and signature of a knot. Both the determinant $\det(K)$ and the signature $\sigma(K)$ of a knot are integer-valued ($\det(K)$ is odd and $\sigma(K)$ is even) and are determined by the Seifert module \cite{GL, Trotter}. A quasi-alternating knot is a generalization of an alternating knot due to \os~\cite{OS:Branched}. All alternating knots are quasi-alternating. Two pairs of knots for which Theorem \ref{thm:sigdif} applies are shown in Figure \ref{fig:567}. Note that any banding relating a pair of knots is necessarily a non-coherent band surgery.

{
\renewcommand{\thetheorem}{\ref{thm:sigdif}}
\begin{theorem}
Let $K$ and $K'$ be a pair of quasi-alternating knots and suppose that $\det(K) = m = \det(K')$ for some square-free integer $m$. If there exists a band surgery relating $K$ and $K'$, then $|\sigma(K) - \sigma(K')|$ is $0$ or $8$. 
\end{theorem}
\addtocounter{theorem}{-1}
}

\begin{figure}
\includegraphics[height=0.23\textwidth]{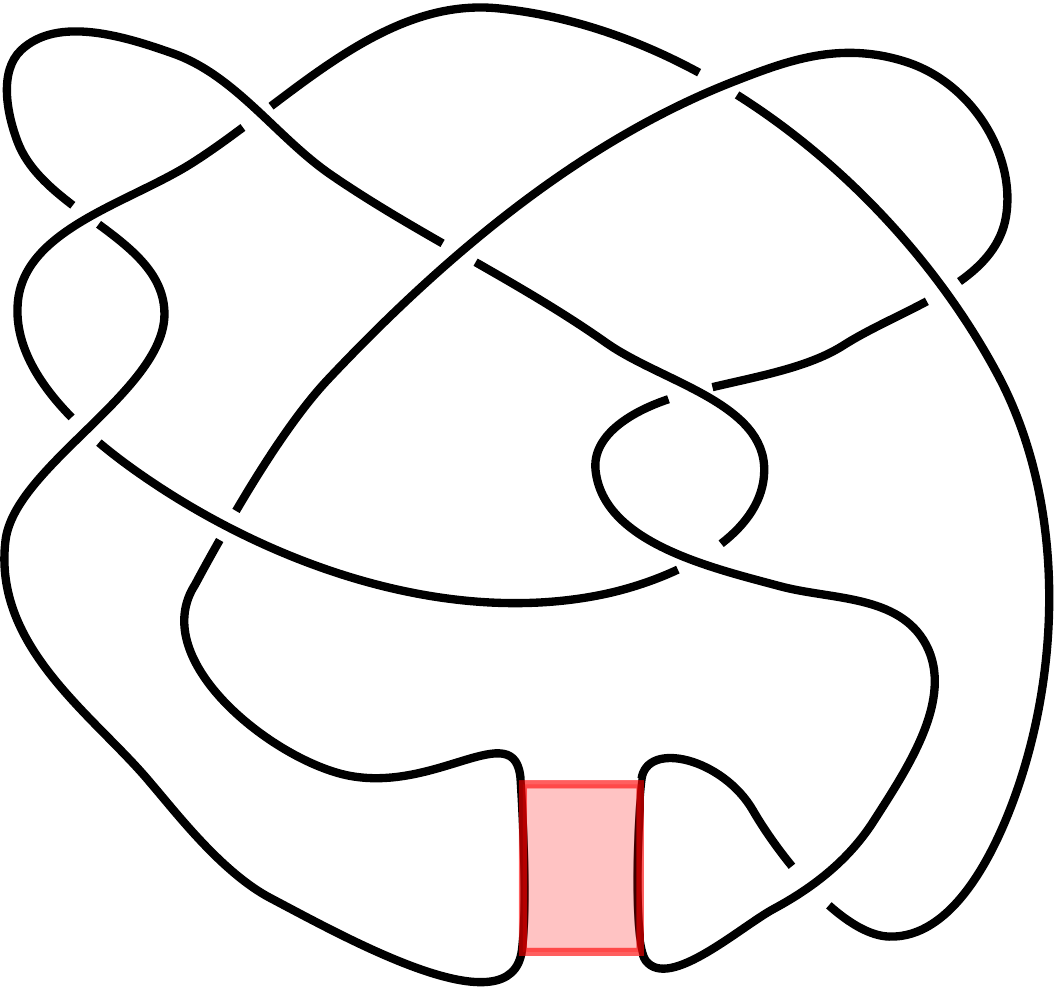}
\qquad  \qquad  \qquad
\includegraphics[height=0.23\textwidth]{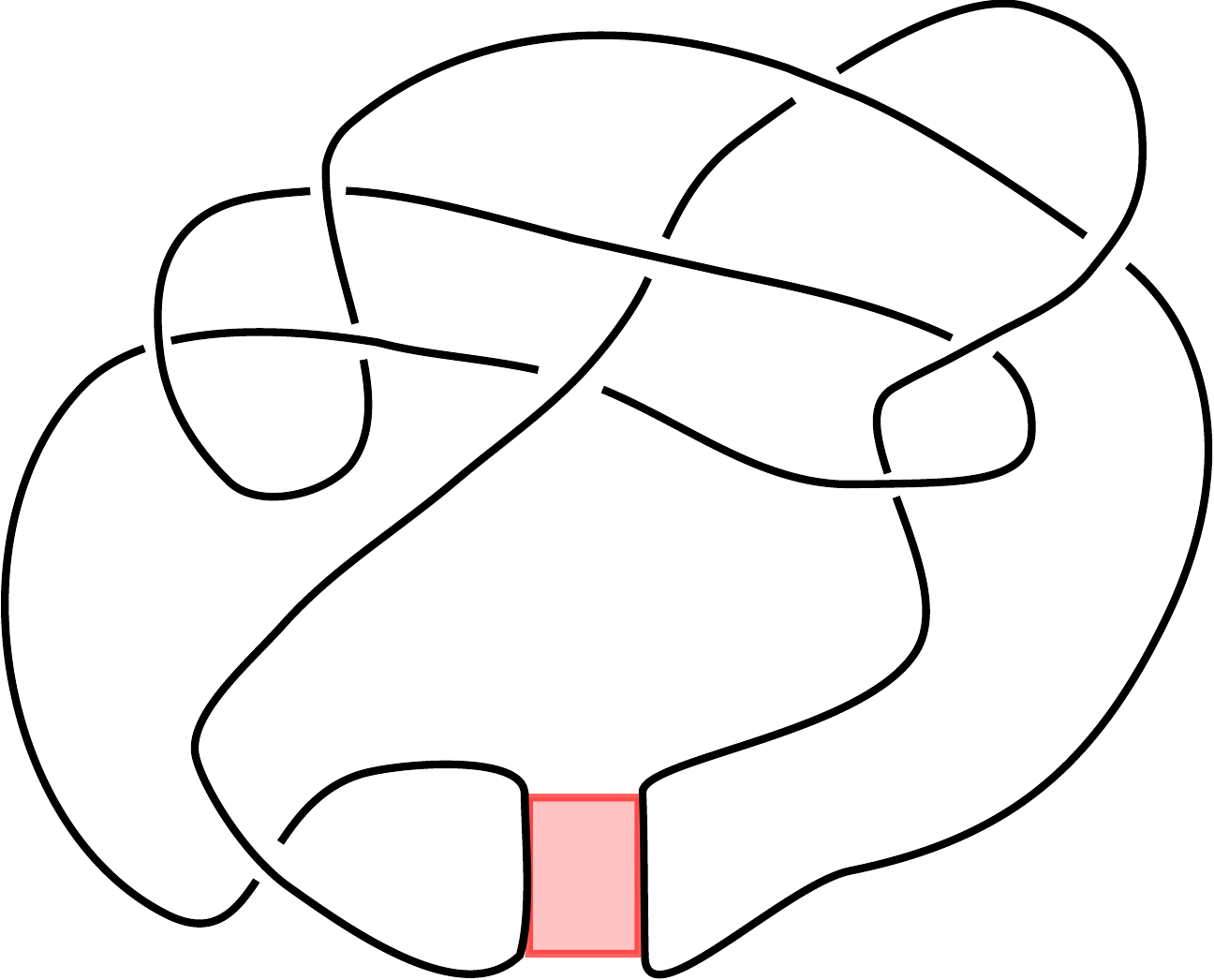}
\caption{Non-coherent band surgery (indicated by the shaded area) relates the knots $7_1$ and $5_2^*$ (left), and the knots $6_2$ and $7_2$ (right). The knot $5_2^*$ is the mirror of $5_2$. Diagrams corresponding with these band moves were found by Zekovi\'c \cite{Zekovic}. In contrast, Theorem \ref{thm:sigdif} implies there cannot exist a band surgery between $7_1$ and $5_2$, since $\det(7_1) = 7=\det(5_2)$ and $|\sigma(7_1)-\sigma(5_2)| =4$. Likewise, there is no band surgery between $6_2$ and $7_2^*$. See example \ref{example:knots} for details.}
\label{fig:567}
\end{figure}

The determinant of a knot, or more generally the first homology of its branched double cover, can often provide an effective obstruction to the existence of a band surgery relating a pair of knots (see for example \cite{AK, KM, Kanenobu}). 
However, when a pair of knots $K, K'$ have branched double covers with isomorphic first homology (which implies $\det(K) = \det(K')$), such obstructions do not apply. 
Theorem \ref{thm:sigdif} fills this gap in the case where $m$ is square-free. 
Theorem \ref{thm:sigdif} is reminiscent of Murasugi's well-known statement that $|\sigma(L) - \sigma(L')|\leq 1$ when $L$ and $L'$ are related by a coherent band surgery \cite[Lemma 7.1]{Murasugi}. 
In the case of non-coherent band surgery, the signature of a knot may change by an arbitrary amount. 
For example, the torus knot $T(2, m)$, with $m>0$ odd, has signature $1-m$ and is related by a single band surgery operation to the unknot, which has zero signature. 

Murasugi also proved that for any knot $K$, the signature controls the determinant of a knot modulo 4 \cite[Theorem 5.6]{Murasugi}. 
More precisely, if $|\sigma(K)| \equiv 0$ (resp. 2) modulo 4, then $|\det(K)|\equiv 1$ (resp. 3) modulo 4. This immediately implies
\[
	|\sigma(K) - \sigma(K')| \equiv |\det(K) - \det(K')| \pmod{4},
\]
for \emph{any} pair of knots, which provides an explanation for the significance of the numbers 0 and 8 in Theorem \ref{thm:sigdif}. 

Theorem \ref{thm:sigdif} follows as a corollary of Theorem \ref{thm:main}, a more general statement about the Heegaard Floer $d$-invariants of pairs of L-spaces related by integral surgery.   

{
\renewcommand{\thetheorem}{\ref{thm:main}}
\begin{theorem}
Let $Y$ and $Y'$ be L-spaces with $H_1(Y) \cong \Z/d_1 \oplus \Z/d_2\oplus \cdots \oplus \Z/d_k \cong H_1(Y')$, where each $d_i$ is an odd square-free integer. If $Y'$ is obtained by a distance one surgery along any knot $K$ in $Y$, then
\[
	| d(Y, \mft) - d(Y', \mft') | = 2 \text{ or } 0
\]
where $\mft$ and $\mft'$ denote the unique self-conjugate $\spinc$ structures on $Y$ and $Y'$.
\end{theorem}
\addtocounter{theorem}{-1}
}

Heegaard Floer homology is a powerful package of three and four-manifold invariants due to \os~\cite{OS:Absolutely}. 
An \emph{L-space} is a rational homology sphere whose Heegaard Floer homology is as simple as possible. One particularly useful Heegaard Floer invariant is the \emph{d-invariant} of the pair $(Y,\mft)$, where $Y$ is an oriented rational homology sphere and $t$ is an element of $\spinc (Y )\cong H^2 (Y; \Z)$. The $d$-invariant is a rational number, defined as the minimal grading of a certain submodule of the Heegaard Floer module $HF^+(Y, \mft)$ \cite{OS:Absolutely}. The reader is referred to \cite{OS:Absolutely} for an introduction to these invariants. All required background for this note is in sections \ref{sec:prelims} and \ref{sec:hf}.

The \emph{distance} between two Dehn surgery slopes refers to their minimal geometric intersection number. 
A surgery slope that intersects the meridian of a knot exactly once is called a \emph{distance one surgery}, or an \emph{integral} surgery. 
Given two knots or links related by band surgery, the Montesinos trick \cite{Montesinos} implies that their branched double covers may be obtained by distance one Dehn fillings of a three-manifold with torus boundary. 
To prove Theorem \ref{thm:main} we apply the Heegaard Floer mapping cone formula of \os~\cite{OS:Rational}, and a theorem of Ni and Wu \cite[Proposition 1.6]{NiWu:Cosmetic} which describes the $d$-invariants of a manifold obtained by integral surgery along a null-homologous knot in an L-space in terms of certain integer-valued knot invariants due to Rasmussen \cite{Rasmussen:Thesis}.  

\begin{figure}
\centering
\includegraphics[height=3.3cm]{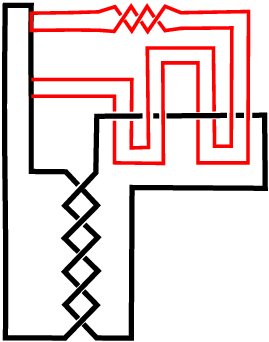}
\qquad \qquad  \qquad
\includegraphics[height=3.7cm]{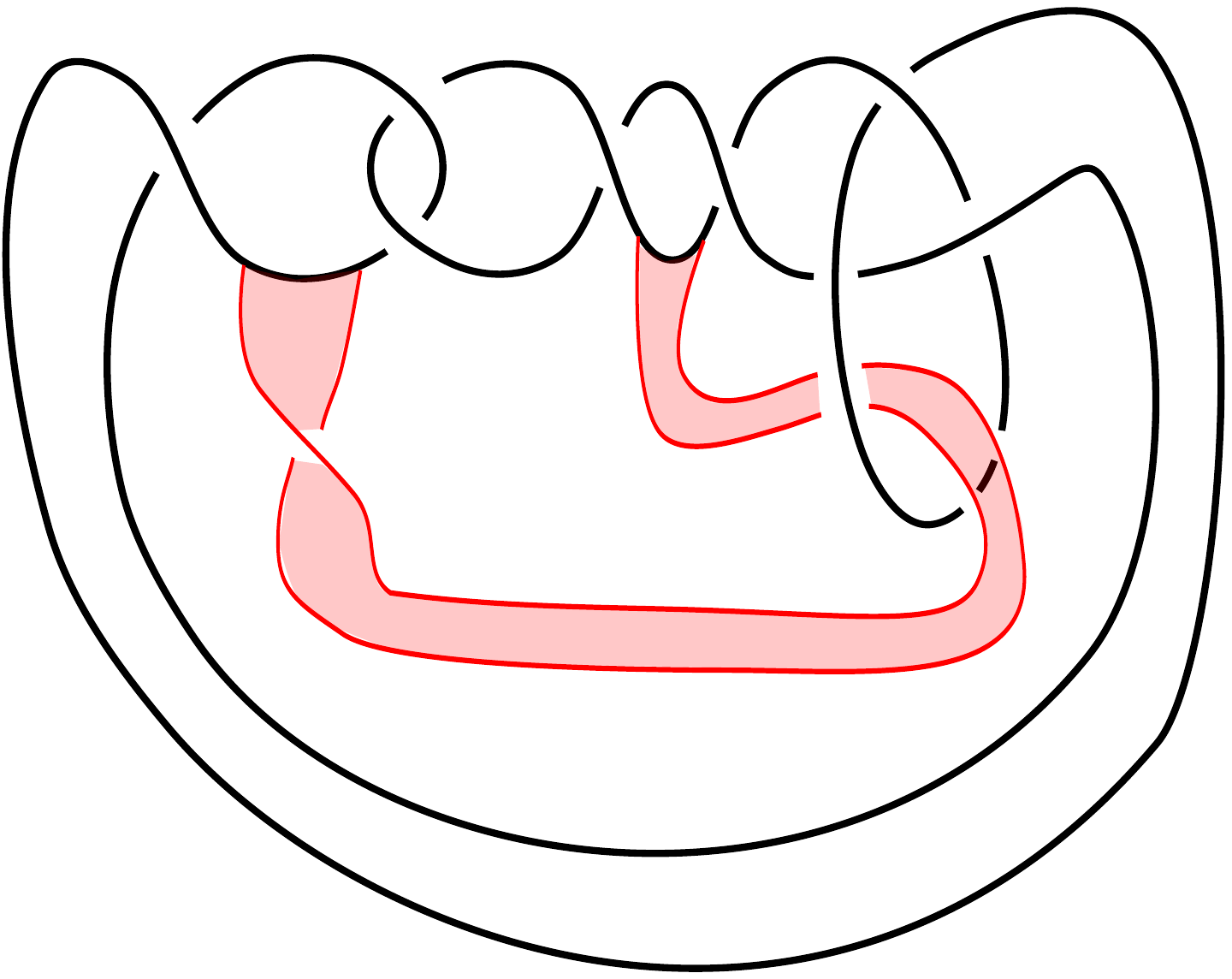}
\caption{(Left) A banding relating the torus knot $T(2, 5) = 5_1$ to its mirror $T(2, -5) = 5_1^*$. The existence of this band move is due to Zekovi\'c \cite{Zekovic}. 
(Right) A banding relating the pair $8_8^*$ and $8_8$. This band was discovered via computer simulations (see section \ref{sec:chiral}).}
\label{fig:chiralexamples}
\end{figure}

\textbf{Applications to chirally cosmetic banding.} A special case of Theorem \ref{thm:sigdif} occurs when a band surgery relates a knot $K$ with its mirror image $K^*$. 
See Figure \ref{fig:chiralexamples} for an example. 
This is called a \emph{chirally cosmetic banding}. 
The following is an immediate consequence of Theorem \ref{thm:sigdif}. 
{
\renewcommand{\thetheorem}{\ref{cor:torus}}
\begin{corollary}
The only nontrivial torus knot $T(2, m)$, with $m$ square free, admitting a chirally cosmetic banding is $T(2, 5)$.
\end{corollary}
\addtocounter{theorem}{-1}
}

In section \ref{sec:band}, we compare the square-free condition of Theorem \ref{thm:sigdif} with the constructive examples of chirally cosmetic bandings given in \cite{IJ}. 
Using Monte Carlo simulations we find examples of knots admitting chirally cosmetic bandings. Additionally, in all cases these appear with low probability.
Amongst knots of up to eight crossings, the results of our simulations produce three such examples: $5_1$, $8_8$, and $8_{20}$. The chirally cosmetic banding along $8_8$ was previously not known to exist. An isotopy of the banding is shown in Figure \ref{fig:chiralexamples} (right). 

\textbf{Organization.} In section \ref{sec:prelims} we introduce notation and establish some prerequisite topological information. Section \ref{sec:hf} contains the necessary background in Heegaard Floer homology along with the proof of Theorem \ref{thm:main}. In section \ref{sec:band} we prove Theorem \ref{thm:sigdif}, further consider the case of chirally cosmetic pairs and detail the methods of the computer simulations. 

\section{Preliminaries}
\label{sec:prelims}

\subsection{Homological preliminaries}
In this section we state some homological preliminaries and set notation. We assume all singular (co)homology groups have coefficients in $\Z$. Let $K$ be a knot in a rational homology sphere $Y$, and let $M=Y-N(K)$ denote the complement of $K$ in $Y$. A \emph{slope} is an isotopy class of an unoriented simple closed curve on the boundary of $M$. The \emph{Dehn filling} $M(\eta)$ is the closed, oriented three-manifold obtained by gluing a solid torus $V$ to $M$ by a homeomorphism which identifies a meridian of $\partial V$ to the slope $\eta$ on $\partial M$. The \emph{distance} between two surgery slopes on $\partial M$ is the minimal geometric intersection number of the two curves, denoted $\Delta(\eta, \nu)$ for any pair of slopes $\eta, \nu$. In this note we are particularly concerned with surgery slopes which intersect the meridian of $K$ exactly once, i.e. \emph{distance one surgeries}, or \emph{integral} surgeries. 

When $K$ is null-homologous, there is a standard choice of a meridian $\mu$ and Seifert longitude $\lambda$ on the bounding torus $\partial M$. Here, the meridian $\mu$ is a simple closed curve on $\partial M = \partial N(K)$ that bounds a disk in the solid torus $N(K)$. The preferred longitude is the slope determined by the intersection of a Seifert surface for the knot and $\partial N(K)$. A slope $\eta$ may be written in terms of this standard basis as $\eta=a\mu+b\lambda$. Rational Dehn surgery with the surgery coefficient $a/b$ along $K$ in $Y$ is then well-defined, and the result of surgery is denoted by $Y_{a/b}(K) = M(\eta)$. Indeed, when $K$ is null-homologous a homological argument (see for example \cite[Lemma 8.1]{Gainullin}) shows that $H_1(M)\cong \Z \oplus H_1(Y)$, where the $\Z$-summand is generated by the meridian of $K$. This implies that $H_1(Y_{a/b}(K) ) \cong \Z/a  \oplus H_1(Y)$ by a standard Mayer-Vietoris argument. Notice that the Seifert longitude controls the first homology of the Dehn filling, in the sense that $|H_1(Y_{a/b}(K))| = |M(\eta)| = |H_1(Y)|\cdot \Delta(\eta, \lambda)$. 

When $K$ is not null-homologous, we have that $H_1(M)\cong \Z \oplus H$ for some finite group $H$.   Let $i_*$ denote the map induced by inclusion in the long exact sequence of the pair $(M, \partial M)$,
\begin{equation*}
\cdots \rightarrow H_2(M) \rightarrow H_2(M,\partial M)\rightarrow H_1(\partial M)\stackrel{i_*}{\rightarrow} H_1(M) \rightarrow H_1(M, \partial M) \rightarrow \cdots.
\end{equation*}
The kernel and image of $i_*$ are both rank one. In particular, the kernel of $i_*$ is generated by $k\lambda_M$, where $\lambda_M$ is a primitive class $[\lambda_M]$ in $H_1(\partial M)$ uniquely determined up to sign. This class determines a well-defined slope in $\partial M$:
\begin{definition}
The \emph{rational longitude} $\lambda_M$ is the unique slope in $\partial M$ characterized by the property that its image $i_*(\lambda_M)$ is of finite order in $H\subset H_1(M)$. 
\end{definition}
When we work with knots that are non-trivial in $H_1(Y)$, we will use the rational longitude to fix a basis $(\mu,\lambda_M)$ for $H_1(\partial M)$, where $\mu$ now denotes some slope with $\Delta(\mu, \lambda_M)=1$. Just as with the case of the Seifert longitude for null-homologous knots, the rational longitude also controls the order of the first homology of the filling. 
\begin{lemma}\cite[Lemma 3.2]{Watson}
\label{lem:watson}
For any filling slope $\eta\neq\lambda_M$ on $\partial M$, 
\begin{equation}
\label{eqn:order}
	|H_1(M(\eta)| = \operatorname{ord}_H i_*(\lambda_M)\cdot|H| \cdot \Delta(\eta, \lambda_M)
\end{equation}
where $\operatorname{ord}_H i_*(\lambda_M)$ denotes the (necessarily finite) order of the rational longitude in $H$. 
\end{lemma}
In \cite[Lemma 3.2]{Watson}, this is proved with a careful analysis of the long exact sequence of the pair. Under the homomorphism $i_*:H_1(\partial M)\rightarrow H_1(M)$, the basis elements $\mu$ and $\lambda_M$ are mapped to $(\ell, u)$ and $(0, h)$, respectively, where $|\ell|=\operatorname{ord}_H i_*(\lambda_M)$, and $u,h$ are some elements of $H$. 

Let us introduce some notation. Consider a decomposition of the finite abelian group $H$ as $\Z/r_1\oplus \cdots \oplus \Z/r_k$. We will write $\vec{u}=(u_1, \cdots, u_k)$ and $\vec{v} = (v_1, \cdots, v_k)$ for elements $u, v\in H$ written with respect to this decomposition. The notation $I_{\vec{r}}$ will stand for the $k\times k$ diagonal matrix with $i$-th diagonal entry $r_i$ where $r_1, \dots, r_k$ are the invariant factors of $H$. 

For a filling slope $\eta = p\mu+q\lambda_M$, Watson observes that $H_1(M(\eta)) = H_1(M)/(p\ell, pu+qh)$ has a block form presentation matrix 
\begin{equation}
\label{pres} 
	\left( \begin{array}{cc}
	p\ell & 0\\
	p\vec{u}+q\vec{h} & I_{\vec{r}}
	\end{array} \right),
\end{equation}
The statement of the lemma then follows. 

The reader may notice that the proof of Lemma \ref{lem:surg} below is adapted from the argument given in \cite[Theorem 2.4]{LM}. In general, we will use the Smith normal form of the presentation matrix \eqref{pres}. Recall that the Smith normal form is a diagonal matrix $I_{\vec{\delta}}$ where $\delta_i|\delta_{i+1}$. Each $\delta_i = \Gamma_i/\Gamma_{i-1}$, where $\Gamma_i$ is the greatest common divisor of the the $i\times i$ minors of $A$, and $\Gamma_0=1$. 

\begin{lemma}
\label{lem:surg}
Suppose that $H_1(Y) \cong \Z/d_1 \oplus \Z/d_2\oplus \cdots \oplus \Z/d_k \cong H_1(Y')$ where each $d_i$ is an odd square-free integer. If $Y'$ is obtained from $Y$ by a distance one surgery on a knot $K$ in $Y$, then $K$ is null-homologous and the surgery coefficient is $\pm 1$.
\end{lemma}

\begin{proof}
Let $M=Y-N(K)$ and fix $(\mu, \lambda_M)$ as a basis of $H_1(\partial M)$, where $\mu$ is some curve dual to the rational longitude $\lambda_M$. In this basis, filling slopes $\alpha$ and $\beta$ yielding $Y$ and $Y'$, respectively, may be written as $\alpha = p\mu +q\lambda_M$ and $\beta =r\mu +s\lambda_M$ for some integers $p,r> 0$ and $q, s$ relatively prime to $p, r$. 

By assumption $|H_1(M(\alpha))|=|H_1(M(\beta))|$. Applying Lemma \ref{lem:watson}, we have
\[
	 \operatorname{ord}_H i_*(\lambda_M)\cdot|H| \cdot \Delta(\alpha, \lambda_M) = |H_1(M(\alpha))| = |H_1(M(\beta))| = \operatorname{ord}_H i_*(\lambda_M)\cdot|H| \cdot \Delta(\beta, \lambda_M) .
\]
This implies that $p = \Delta(\alpha, \lambda_M) =\Delta(\beta, \lambda_M) =r$. The assumption that $Y'$ is obtained by a distance one surgery along $K$ means that $\Delta(\alpha, \beta)$ is one. Therefore
\[
 1 = \Delta(\alpha, \beta)= p(q-s),
\]
and so $|p|=1$ and $|q-s|=1$. After possibly multiplying by $-1$ to ensure $p=1$, we may write $\alpha = \mu +q\lambda_M$ and $\beta =\mu  +(q\pm1)\lambda_M$. With the change of basis $(\mu, \lambda_M)\mapsto (\mu+q\lambda_M, \lambda_M)$, we can further assume $\alpha=\mu$ and $\beta =\mu \pm \lambda_M$.

By \eqref{eqn:order}, we have that
\[
	|H_1(Y)| = |H_1(M(\alpha))| = \operatorname{ord}_H i_*(\lambda_M)\cdot|H|.
\]
If $\operatorname{ord}_H i_*(\lambda_M)=1$, then $\lambda_M$ bounds in $M$. Because $\lambda_M$ is homologous to the core of the filling torus $N(K)$ in $M(\alpha)$, this also means that $K$ is null-homologous in $Y=M(\alpha)$. Thus we aim to show that $\operatorname{ord}_H i_*(\lambda_M)=1$.

Suppose now it is not the case, so that $\operatorname{ord}_H i_*(\lambda_M)=|\ell|\neq1$. 
Consider the presentation matrices $A$ and $B$ for $H_1(M(\alpha))$ and $H_1(M(\beta))$,
\begin{equation*}
	A=\left( \begin{array}{ccccc}
	\ell &  \\
	u_1 & r_1 \\
	\vdots && \ddots \\
	u_k & & & r_k
	\end{array} \right) \text{ and }
	B=\left( \begin{array}{ccccc}
	\ell &  \\
	u_1\pm h_1 & r_1 \\
	\vdots && \ddots \\
	u_k \pm h_k & & & r_k
	\end{array} \right),
\end{equation*}
where we have multiplied the first column by a unit to ensure that $\ell=\operatorname{ord}_H i_*(\lambda_M)$. 

We claim that each $u_i$ is a multiple of $\gcd(\ell, r_i)$. If not, then there is some $u_i$ for which there exists a prime power $p^j$ that divides both $\ell$ and $r_i$, but not $u_i$. Consider now $\det(A)=|\ell r_1\cdots r_k| = |\delta_1\cdots\delta_{k+1}|$. Since $p|\ell$ and $p|r_i$ we know that $p^t|\det(A)$ for some power $t$. However, if we consider the $k\times k$ minor $A_{1, i+1}$ we see (by, say, cofactor expansion at $u_i$) that $\det(A_{1, i+1})=|r_1\cdots r_{i-1}u_ir_{i+1}\cdots r_k|$ is divisible by at most $p^{t-2}$. Consider now the Smith normal form of $A$. By definition $\delta_{k+1}=\det(A)/\Gamma_k$. In particular, $\Gamma_k$ is divisible by at most $p^{t-2}$. But this implies $p^2|\delta_{k+1}$, and hence some invariant factor is not square-free. Because we have assumed that the $d_i$ are square-free, we must also have that the $\delta_i$ are square free. Hence we have reached a contradiction. 

Notice that the same argument will apply to the presentation matrix $B$ for $H_1(M(\beta))$ to show that $\gcd(\ell, r_i)|(u_i\pm h_i)$ for $i=1, \cdots, k$. Since $\gcd(\ell, r_i)|(u_i\pm h_i)$ and $\gcd(\ell, r_i)|u_i$, we also have $\gcd(\ell, r_i)|h_i$. With this, we may write
\begin{equation*}
	B=\left( \begin{array}{ccccc}
	\ell &  \\
	u_1+ b_1\gcd(\ell, r_1) & r_1 \\
	\vdots && \ddots \\
	u_k+ b_k\gcd(\ell, r_k) & & & r_k
	\end{array} \right),
\end{equation*}
for some integers $b_1, \cdots, b_k$. 

Observe that $h\in H$ is trivial if and only if $\vec{h}$ is in the column space of $I_{\vec{r}}=B_{1,1}$. In particular, since $\ell=\operatorname{ord}_H i_*(\lambda_M)$, we have that $\ell \vec{h}$ is in the column space of $B_{1,1}$. Thus there exist integers $x_i$ so that for all $i=1, \cdots, k$.  
\[
	\ell h_i = \ell b_i \gcd(\ell, r_i) = x_i r_i.
\]
This then implies $r_i /\gcd(\ell, r_i)$ is a divisor of $b_i$, so that $b_i = y_i r_i /\gcd(\ell, r_i)$ for some $y_i$. We can then write $h_i = y_i r_i$. This shows $\vec{h}$ is in the column space of $B_{1,1}$, which implies $i_*(\lambda_M) = (0, h)\in \Z\oplus H$ is trivial.

We have established that $\operatorname{ord}_H i_*(\lambda_M)=1$ and hence $K$ is null-homologous in $Y=M(\alpha)$. Now we take may take the preferred basis $(\mu, \lambda)$ on $H_1(\partial M)$, where $\mu$ is the meridian of $K$ and $\lambda$ is a Seifert longitude. 
The condition that the slope yielding $Y'$ is distance one from the meridian implies the slope must in fact be $\mu\pm\lambda$, i.e. the surgery coefficient is $\pm1$.\qedhere
\end{proof}

For the remainder of the article, we will assume that $K$ is a null-homologous knot in $Y$ with meridian $\mu$ and longitude $\lambda$. We will be primarily concerned with slopes of the form $n\mu+\lambda$, for $n\in\Z$, which are framing curves on the boundary of a neighborhood of $K$. These slopes intersect the meridian once transversely and inherit an orientation from $K$. The notation $Y_n(K)$ will denote the result of integral surgery along $K$, and $K_n$ will denote the core of the surgery $Y_n(K)$.

The set $\spinc(Y)$ is the space of nowhere vanishing vector fields on $Y$ modulo homotopy outside of a ball. There is a non-canonical correspondence $\spinc(Y)\cong H^2(Y)$ \cite{turaev}.  We write $\mft$ to denote an element of $\spinc(Y)$, and write $\relspinc(Y, K)$ to denote the relative $\spinc$ structures on $(M,\partial M)$. Note that $\relspinc(Y, K)$ has an affine identification with $H^2(Y,K)$, which by excision is isomorphic with $H^2(M,\partial M)\cong H_1(M)$. In particular, because $K$ is null-homologous and $H_1(M)\cong H_1(Y) \oplus \Z$, we may label an element $\xi$ of $\relspinc(Y, K)$ by $\xi=(\mft, s)$, where $\mft\in\spinc(Y)$ and $s\in\Z$. The correspondence $\relspinc(Y, K)\equiv H^2(M,\partial M)$ is non-canonical, and the labelling that we adopt follows closely to that of \cite{Gainullin}. When $H_1(Y)$ has odd order, there is a unique self-conjugate $\spinc$ structure $\mft_0$, distinguished from the others by the requirement that it satisfies $c_1(\mft_0) = 0 \in H^2(Y)$ (see for example [Section 2]\cite{MO:Concordance}).

\section{Heegaard Floer invariants and the proof of Theorem \ref{thm:main}}
\label{sec:hf}
\subsection{Heegaard Floer background and the mapping cone formula}
\label{sec:hfback}

We will assume some familiarity with Heegaard Floer homology, referring the reader to \cite{OS:Absolutely, OS:Rational} for more information. We take all Heegaard Floer complexes with coefficients in the field $\F=\Z/2$.  One of the main components of the Heegaard Floer package we will use is the $d$-invariant $d(Y, \mft)$, or \emph{correction term}, which is a rational number associated to a $\spinc$ rational homology sphere $(Y, \mft)$. More specifically, given a rational homology sphere $Y$, the Heegaard Floer module $HF^+(Y)$ splits over $\spinc$ structures, and in each summand we have
\begin{equation}
\label{eq:hfsum}
	HF^+(Y, \mft) = \F[U, U^{-1}]/ U \cdot \F[U] \oplus HF^+_{red}(Y, \mft),
\end{equation}
as $\F[U]$-modules. The first summand is abbreviated $\T_d^+$ and referred to as the \emph{tower}. The second summand is a torsion $\F[U]$-module. An \emph{L-space} is a rational homology sphere whose Heegaard Floer homology $HF^+(Y)$ is a free $\F[U]$-module with rank $|H^2(Y;\Z)|$. That is, the torsion summand in \eqref{eq:hfsum} vanishes, leaving only a tower in each $\spinc$ structure. The $d$-invariant $d:=d(Y,\mft)$ is defined to be the minimal Maslov grading of the tower. The $d$-invariants switch sign under orientation-reversal \cite{OS:Absolutely}.

Associated to $(Y, K)$ and each relative $\spinc$ structure $\xi$ in $\relspinc(Y,K)$ is the knot Floer chain complex $C_\xi = CFK^\infty(Y, K, \xi)$ \cite{Rasmussen:Thesis, OS:KnotInvariants, OS:Rational}. The complex is $\Z$-filtered over $\F[U, U^{-1}]$ with a second filtration induced by the action of the variable $U$. We denote these two filtrations as $(i, j) =$ \emph{(algebraic, Alexander)}. The chain complex also has a homological Maslov grading that we suppress in the notation. Multiplication by $U$ decreases the Maslov grading by two and the Alexander filtration by one, and the action of $PD[\mu]$ on $\xi$ shifts the Alexander filtration by one, i.e. $C_{\xi+PD[\mu]} = C_\xi [(0, -1)]$. Recall that $\mu$ denotes the (class of a) meridian.

As in \cite{OS:Rational}, for each $\xi$ in $\relspinc(Y,K)$ there are complexes $\Aplus_\xi = C_\xi\{\max\{i,j\} \geq 0\}$ and $\Bplus_\xi = C_\xi\{i \geq 0\}$.  The complex $\Bplus_\xi$ is $CF^+(Y,G_{Y,K}(\xi))$. The map $G_{Y,K}:  \relspinc(Y, K) \rightarrow \spinc(Y)$ is defined in \cite[Section 2.2]{OS:Rational}, and sends a relative $\spinc$ structure to a $\spinc$ structure in the target manifold indicated by the subscript. 
The complex $\Aplus_\xi$ represents the Heegaard Floer complex of a large surgery $Y_N(K)$ in a certain $\spinc$ structure, where $Y_N(K)$ is obtained by Dehn surgery along a framing curve $N\mu+\lambda$ with $N>>0$. 
There are analogous complexes in the `hat' version of Heegaard Floer homology. In particular we have $\Ahat_\xi = C_\xi\{\max\{i,j\} = 0\}$ and $\Bhat_\xi = C_\xi\{i = 0\} \cong \CFhat(Y,G_{Y,K}(\xi))$. The complexes are related by chain maps
\begin{equation}
\label{eqn:vhmaps}
\begin{array}{r@{} r@{\qquad} l}
	\vplus_\xi: \Aplus_\xi \to \Bplus_\xi, & &\hplus_\xi : \Aplus_\xi \to \Bplus_{\xi + PD[K_n]},\\ 
	\vhat_\xi: \Ahat_\xi \to \Bhat_\xi, &  &\hhat_\xi : \Ahat_\xi \to \Bhat_{\xi + PD[K_n]},
\end{array}
\end{equation}
where $K_n$ is the push-off of $K$ inside $Y-N(K)$ using framing $n\mu+\lambda$. 
In \cite[Theorem 4.1]{OS:Rational}, it is shown that the maps $\vplus_\xi$ and $\hplus_\xi$ correspond with a negative definite cobordism $W'_N:Y_N(K) \rightarrow Y$ equipped with the $\spinc$ structures $\mfv$ and $\mfh = \mfv+PD[\Fhat]$, respectively, which extend a given $\spinc$ structure on $Y_N(K)$. Here, $[\Fhat]$ generates $H_2(W'_N, Y)$ and is represented by a capped-off Seifert surface for $K$. 

The complexes $\Aplus_\xi$ and $\Bplus_\xi$ each contain a non-torsion summand, i.e. a tower. On homology, each of the maps $\vplus_\xi$ and $\hplus_\xi$ induces an endomorphism of the towers, which is multiplication by $U^V_\xi$ or $U^H_\xi$ for integers $V_\xi\geq 0$ and $H_\xi\geq0$. This defines the knot invariants $V_\xi$ and $H_\xi$, which appeared originally as the \emph{local h-invariants} in \cite{Rasmussen:Thesis}. We remark for later use that when $V_\xi > 0$, the corresponding map $\vhat_\xi$ is identically zero, and similarly $H_\xi > 0$ implies $\hhat_\xi$ is zero.

Recall that because $K$ is null-homologous, $\relspinc(Y,K) \cong \spinc(Y)\oplus \Z$, where the $\Z$-summand is generated by the meridian of $K$ \cite[Lemma 8.1]{Gainullin}. We write $\xi = (\mft, s)$ for $\mft\in \spinc(Y)$ and $s\in\Z$. Fixing $\mft\in\spinc(Y)$, we have the subgroup $\{\mft \} \oplus \Z\cong \Z$. Because the following properties have appeared in various forms throughout the literature, we provide just a sketch of the proof.

\begin{property} 
\label{prop:vh}
Let $K$ be a null-homologous knot in a rational homology sphere $Y$ and let $\mft$ be a self-conjugate $\spinc$ structure on $Y$. 
Then the invariants $V_{(\mft, s)}$ and $H_{(\mft, s)}$ satisfy:
\begin{enumerate}
	\item\label{item-VandHspec} $V_{(\mft, s)} \geq V_{(\mft, s+1)} \geq V_{(\mft, s)} -1$, 
	\item\label{item-V=H} $V_{(\mft, s)} = H_{(\mft, -s)}$,
	\item\label{item-H>V} $H_{(\mft, s)} \geq V_{(\mft, s)}$ for $s\geq0$.
\end{enumerate}
\end{property}
\begin{proof}[(sketch)]
Since  $\relspinc(Y,K) \cong \spinc(Y)\oplus \Z$, we have that $\xi+PD[\mu] = (\mft, s)+PD[\mu] = (\mft, s+1)$. When we fix a self-conjugate $\spinc$ structure on $\mft$ on $Y$, the statements above relating $V_{(\mft, s)}$ and $H_{(\mft, s)}$ are analogous to those for the invariants $V_s$ and $H_s$, with $s\in\Z$, for the case of a knot in an integer homology sphere. In particular, Property \ref{item-VandHspec} follows as a direct analogue of \cite[Property 7.6]{Rasmussen:Thesis} (see also \cite[Lemma 2.4]{NiWu:Cosmetic}). 
Next, because conjugation changes the sign of the first Chern class and the $H^2$-action, we may verify using the formulas of \cite[Section 4.3]{OS:Integer} that $\mfh_{(\mft, s)}$ and $\mfv_{(\mft, s)}$ are conjugate $\spinc$ structures on the four-manifold cobordism $W'_N$. This implies Property \ref{item-V=H}. Finally, \ref{item-VandHspec} and \ref{item-V=H} imply \ref{item-H>V}. See also \cite[Lemmas 2.3-2.5]{HLZ}.
\end{proof}

The proof of Theorem \ref{thm:main} below will require the mapping cone formula for the Heegaard Floer homology of the integral surgery $Y_n(K)$. We give only a terse review here, sending the reader to \os~for details \cite{OS:Integer, OS:Rational}. The generalization of the mapping cone formula specific to rational surgeries on null-homologous knots in L-spaces is also reviewed in \cite{NiWu:Cosmetic, Gainullin}. Note that for our applications, we only require the formulation for integral surgery and `hat' homology.

For each $\mft \in \spinc(Y)$ and $0 \leq i < n$, sum up the complexes $\Ahat_\xi=\Ahat_{(\mft, s)}$ and $\Bhat_\xi=\Bhat_{(\mft, s)}$ into
\[
	\widehat{\mathbb{A}}_{(\mft, i)} = \bigoplus_{s\in\Z} (s, \Ahat_{(\mft, i+ns)} )
%\]
\text{ and } 
%\[
	\widehat{\mathbb{B}}_{(\mft, i)} = \bigoplus_{s\in\Z} (s, \Bhat_{(\mft, i)} ),
\]
where the first component of each tuple records the index of the summand. Using the maps $\vhat_\xi$ and $\hhat_\xi$ from equation \eqref{eqn:vhmaps} above, we define the map $\widehat{\mathbb{D}}_{(\mft, i)}:\widehat{\mathbb{A}}_{(\mft, i)} \rightarrow \widehat{\mathbb{B}}_{(\mft, i)}$ by
\[
	\widehat{\mathbb{D}}_{(\mft, i), n}(s, a_s) = (s, \vhat_{(\mft, i+ns)}(a_s)) + (s+1, \hhat_{(\mft, i+ns)}(a_s)).
\]
The mapping cone complex of $\widehat{\mathbb{D}}_{(\mft, i), n}$ is denoted $\Xhat_{(\mft, i),n}$. Let us omit the surgery coefficient `$n$' and write the summand of the cone corresponding to the equivalence class of $\xi=(\mft, i)$ more concisely as $\Xhat_\xi$. 
\begin{theorem}[\os, \cite{OS:Rational}]\label{thm:mappingcone}
Let $\xi \in \relspinc(Y,K)$.  Then there is a relatively-graded isomorphism of groups  
\begin{equation}\label{eq:mappingcone}
H_*(\Xhat_\xi) \cong \widehat{HF}(Y_n(K), G_{Y_n(K),K_n}(\xi)). \\
\end{equation}
\end{theorem}

Given a $\spinc$ structure in $\spinc(Y_n(K))\cong \spinc(Y) \oplus \Z/n$, there is a unique $\spinc$ structure on $Y$ which extends over the two-handle cobordism $W_n(K): Y \rightarrow Y_n(K)$ and this defines the projection from $\spinc(Y_n(K))$ to $\spinc(Y)$. This projection appears in the following result of Ni and Wu, which will allow us to describe the $d$-invariants of surgeries in terms of the invariants $V_{\mft,i}$, which crucially, are non-negative integers. 

\begin{proposition}[Proposition 1.6 in \cite{NiWu:Cosmetic}]
\label{prop:ni-wu}
Fix an integer $n > 0$ and a self-conjugate $\spinc$ structure $\mft$ on an L-space $Y$. Let $K$ be a null-homologus knot in $Y$.   Then, there exists a bijective correspondence $i \leftrightarrow \mft_i$ between $\Z/n$ and the $\spinc$ structures on $\spinc(Y_n(K))$ that extend $\mft$ over $W_n(K)$ such that
	\begin{equation}
	\label{ni-wu-formula}
	d(Y_n(K),\mft_i) = d(Y,\mft) + d(L(n,1),i) - 2N_{\mft, i} 
	 \end{equation} 
where $N_{\mft,i} = \max\{V_{\mft,i}, V_{\mft,n-i}\}$.  Here, we assume that $0 \leq i < n$.  
\end{proposition}
Proposition \ref{prop:ni-wu} was originally proved for knots in the three-sphere, but it generalizes immediately to the case of a null-homologous knot in an L-space \cite{NiWu:Cosmetic}, which is the version stated here (see also \cite{LMV}). 
The term on the right includes the $d$-invariants of the lens space $L(n, 1)$, which are made explicit in section \ref{sec:lens}. 

The following proposition will be useful in the proofs of the main results. 
A proof of Proposition \ref{prop:selfconj} is given in \cite[Proposition 2.10]{LMV}.
\begin{proposition}[Lidman-Moore-Vazquez \cite{LMV}]
\label{prop:selfconj}
Let $K$ be a null-homologous knot in a $\Z/2$-homology sphere $Y$. Let $\mft$ be the self-conjugate $\spinc$ structure on $Y$, and let $\mft_0$ be the $\spinc$ structure on $Y_n(K)$ described in Proposition \ref{prop:ni-wu}. Then $\mft_0$ is also self-conjugate on $Y_n(K)$.
\end{proposition}

\subsection{Proof of Theorem \ref{thm:main}}
\label{sec:hfproof}

In this section we prove the main result. The requirement that $Y$ and $Y'$ are L-spaces in the statement of Theorem \ref{thm:main} is necessary both to calculate the differences in their $d$-invariants via the surgery formula of Proposition \ref{prop:ni-wu}, and to use the surgery formula directly to determine the value of the invariant $V_{\mft,0}$. The L-space condition will surface again when we apply Theorem \ref{thm:main} to alternating and quasi-alternating knots to obtain Corollary \ref{cor:sigdif} and Theorem \ref{thm:sigdif}. The significance in this context is that quasi-alternating knots have branched double covers that are L-spaces, and there is a connection between the signature of the knot and a certain $d$-invariant of the branched double cover.
\begin{theorem}
\label{thm:main}
Let $Y$ and $Y'$ be L-spaces with $H_1(Y) \cong \Z/d_1 \oplus \Z/d_2 \oplus \cdots \oplus \Z/d_k \cong H_1(Y')$, where each $d_i$ is an odd square-free integer. If $Y'$ is obtained by a distance one surgery along any knot $K$ in $Y$, then
\begin{equation}
\label{eq:ddif}
	| d(Y, \mft) - d(Y', \mft') | = 2 \text{ or } 0
\end{equation}
where $\mft$ and $\mft'$ denote the unique self-conjugate $\spinc$ structures on $Y$ and $Y'$.
\end{theorem}

\begin{proof}
Suppose that the homological condition on $Y$ and $Y'$ is satisfied and that $Y'$ is obtained by a distance one surgery along some knot $K$ in $Y$. By Lemma \ref{lem:surg}, $K$ is null-homologous in $Y$ and the surgery coefficient is $\pm1$. 

Let us first consider when the surgery coefficient is exactly $n=+1$, that is, the filling slope yielding $Y'$ is $\mu+\lambda$.  
Since $K$ is a null-homologous knot in an L-space $Y$ along which there exists an integral surgery to $Y'$, we may apply Proposition \ref{prop:ni-wu}. In particular, fixing the unique self-conjugate $\spinc$ structure $\mft$ on $Y$, there is one $\spinc$-structure $\mft_0$ on $\spinc(Y') \cong \spinc(Y) \cong\Z/m$ that extends $\mft$ over $W_1(K)$. By Proposition \ref{prop:selfconj}, $\mft_0$ is self-conjugate on $Y'$. But there is a unique self-conjugate $\spinc$ structure $\mft'$ in $Y'$. Therefore $\mft_{0}=\mft'$ and by equation \eqref{ni-wu-formula} we have
\begin{equation}
	d(Y',\mft') = d(Y,\mft) + d(L(1,1),0) - 2N_{\mft, 0}.
\end{equation} 
Because $L(1, 1)$ is the three-sphere, the term $d(L(1, 1), 0)$ vanishes, leaving
\begin{equation}
	d(Y,\mft) - d(Y',\mft') = 2N_{\mft, 0},
\end{equation} 
where $N_{\mft,0} = \max\{V_{\mft,0}, V_{\mft,1}\}$. We claim that $V_{\mft, 0}$ is at most one. 

Consider the mapping cone formula for the Heegaard Floer chain complex of $Y'$. Recall that $Y' = Y_n(K)$, where we have assumed that $n = +1$. By assumption, both $Y$ and $Y'$ are L-spaces. Therefore $H_*(\Bhat_\xi) \cong \F$, and $H_*(\Xhat_\xi) \cong \F$ for any $\xi \in\relspinc(Y, K)$. 

By \cite[Proposition 2.1]{OS:Lens}, any sufficiently large surgery $Y_N(K)$ will also be an L-space. Since the Heegaard Floer complex of the large surgery $Y_N(K)$ (in some $\spinc$-structure) is quasi-isomorphic with the complex $\Ahat_\xi$, we have that the homology of every summand $\Ahat_\xi$ is torsion-free. In particular,
\begin{equation}
\label{eq:Axi=T}
H_*(\Ahat_\xi) \cong \F \text{ for all } \xi \in \relspinc(Y,K).
\end{equation}
Thus, the Heegaard Floer homology of $Y_n(K)$ is completely determined by the numbers $V_\xi$ and $H_\xi$ for each $\xi \in \relspinc(Y,K)$ (for this statement, $n$ need not be one). 

Since  the mapping cone splits over $\spinc(Y)$, we may restrict our attention to the unique self-conjugate $\spinc$ structure $\mft$ on $Y$. We have that the `hat version' of the mapping cone formula specialized to $n=+1$ surgery along $K$ and restricted to $\mft$ on $Y$ is given by
\[
\xymatrix{
\ldots \ar[drr] & &   \Ahat_{(\mft , 0)} \ar[d]^{\vhat_{(\mft , 0)}}  \ar[drr]^{\hhat_{(\mft , 0)}} & & \Ahat_{(\mft , 1)} \ar[d]^{\vhat_{(\mft , 1)}} \ar[drr]^{\hhat_{(\mft , 1)}} & &\Ahat_{(\mft , 2)} \ar[d]^{\vhat_{(\mft , 2)}} \ar[drr]^{\hhat_{(\mft , 2)}} & & \ldots \\
\ldots & & \Bhat_{(\mft , 0)} & & \Bhat_{(\mft , 1)} & & \Bhat_{(\mft , 2)} & & \ldots 
}
\]
where we have written $\xi = (\mft , s)$ in the above diagram for clarity.

Now suppose that $V_{(\mft , 0)}\geq 2$. Property \ref{prop:vh}\eqref{item-VandHspec} implies that $V_{(\mft , 1)}\geq 1$. Property \ref{prop:vh}\eqref{item-V=H} implies that $H_{(\mft , 0)}\geq 2$, and Property \ref{prop:vh}\eqref{item-H>V} then implies that $H_{(\mft , 1)}\geq 1$. Together these imply that the maps $\vhat_{(\mft , 0)}, \vhat_{(\mft , 1)},  \hhat_{(\mft , 0)}$ and $\hhat_{(\mft , 1)}$ are zero. In particular, $\Ahat_{(\mft , 0)}$ and $\Ahat_{(\mft , 1)}$ are in the kernel of $(\widehat{\mathbb{D}}_\xi)_*$, and so 
\[
	\text{rank } \ker((\widehat{\mathbb{D}}_\xi)_*) \geq 2. 
\]
By \cite[Lemma 12]{Gainullin:Mapping}, the map induced by $\widehat{\mathbb{D}}_{(\mft, i)}$ on homology is surjective. Thus the long exact triangle
\begin{equation}
\label{eq:triangle}
\xymatrix @M=6pt@C=12pt@R=4pt {
	H_*(\widehat{\mathbb{A}}_\xi) \ar[rr]^{(\widehat{\mathbb{D}}_\xi)_*} & &
	H_*(\widehat{\mathbb{B}}_\xi) \ar[dl] \\
	& H_*(\widehat{\mathbb{X}}_\xi) \ar[ul]
}
\end{equation}
implies that $\ker((\widehat{\mathbb{D}}_\xi)_*) \cong H_*(\Xhat_\xi)$. But now \eqref{eq:Axi=T} and \eqref{eq:triangle} imply that
\[
	\text{rank } \ker((\widehat{\mathbb{D}}_\xi)_*) = \text{rank } H_*(\mathbb{X}^+_\xi) = \text{rank } \widehat{HF}( Y', \mft ')\geq 2,
\] 
which contradicts that $Y'$ is an L-space. Therefore $V_{\mft , 0}$ is at most one. This verifies the claim. 

Finally, since $N_{\mft,0} = \max\{V_{\mft,0}, V_{\mft,1}\} = V_{\mft,0}$, this completes the proof of the theorem in the case that the surgery coefficient is $n=+1$. 

If instead the surgery coefficient is $n= -1$, then we may reverse the roles of $Y$ and $Y'$. In particular, we consider $n = +1$ surgery along a null-homologous knot $K_1$ in $Y'$ yielding the manifold $Y$ (here $K_1$ is the core of the previous surgery).  The above argument applies with the corresponding change in notation. 
\end{proof}

\subsection{Lens spaces and $d$-invariants}
\label{sec:lens}
The $d$-invariants of lens spaces can be computed with the following recursive formula of \os.

\begin{theorem}[Proposition 4.8 in \cite{OS:Absolutely}]
\label{thm:d-lens}
Let $p>q > 0$ be relatively prime integers.  Then, there exists an identification $\spinc(L(p,q)) \cong \Z/p$ such that 
\begin{equation}
\label{eq:d-lens}
d(L(p,q),i) = -\frac{1}{4} + \frac{(2i+1-p-q)^2}{4pq} - d(L(q,r),j)
\end{equation}
for $0 \leq i < p+q$.  Here, $r$ and $j$ are the reductions of $p$ and $i \pmod{q}$ respectively.  
\end{theorem}
It is well-known (see for example \cite[Section 3.4]{OwensStrle}), that the self-conjugate $\spinc$ structures on the lens space $L(p, q)$ correspond with the set
\begin{equation}
\label{eq:spin-formula}
\Z \cap \{ \frac{p + q - 1}{2} \text{ , } \frac{q - 1}{2} \}.
\end{equation}
Note that when $p$ is odd, there is a unique self conjugate $\spinc$ structure. 

\begin{corollary} 
\label{cor:lens}
Suppose that $m>0$ is a square-free odd integer. There exists a distance one surgery along any knot $K$ in $L(m, 1)$ yielding $-L(m, 1)$ if and only if $m=1$ or $m=5$.
\end{corollary}

\begin{proof}
Suppose first that $m>0$ is a square-free odd integer and there exists a distance one surgery along any knot $K$ in $L(m, 1)$. Observe that $H_1(L(m, 1))\cong \Z/m$. Because the $d$-invariants change sign under orientation-reversal \cite{OS:Absolutely}, Theorem \ref{thm:main} implies that $d(L(m,1),\mft_0) = 0$ or $1$. By equation \eqref{eq:spin-formula}, the $\spinc$ structure $\mft_0$ corresponds with 0, hence by equation \eqref{eq:d-lens} in Theorem \ref{thm:d-lens} we have
\[
	d(L(m,1),0) = \frac{m-1}{4}.
\]
The $d$-invariant is equal $0$ for $m=1$ and equal $1$ for $m=5$. 

For the reverse direction, note that the branched double cover of $T(2, m)$ is the lens space $L(m, 1)$, and that $L(1, 1)$ is the three-sphere. The Montesinos trick implies that any banding along $T(2, m)$ lifts to a distance one Dehn surgery in $L(m, 1)$. In the case $m=5$, a chirally cosmetic banding from the torus knot $T(2, 5)$ to its mirror image $T(2,-5)$ exists \cite{Zekovic} and is pictured in Figure \ref{fig:chiralexamples}. This band move lifts to a distance one filling taking $L(2, 5)$ to $-L(2, 5)$. In the case $m=1$, the relevant banding is induced by a Reidemeister-I move along an unknot (see Figure \ref{reid}).
\end{proof}
\begin{figure}
\includegraphics[width=1in]{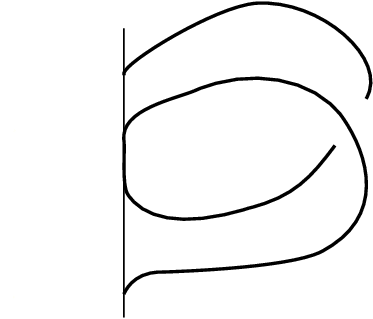}
\caption{A non-coherent banding induced by a Reidemeister-I move.}
\label{reid}
\end{figure}
Corollary \ref{cor:lens} is used in section \ref{sec:band} to show that the only nontrivial torus knot $T(2, m)$ with $m$ square free admitting a chirally cosmetic banding is $T(2, 5)$.

\section{Band surgery and chirally cosmetic bandings}
\label{sec:band}

\subsection{Band surgery} 
\label{subsec:bandsurg}

A band surgery can be described as a two-string tangle replacement as follows. After possibly isotoping $L$ and $L'$, the pairs $(S^3, L)$ and $(S^3, L')$ can be decomposed as 
\[
	(S^3, L)= (B_o, t_o) \cup (B, t) \qquad\text{ and }  \qquad (S^3, L')= (B_o, t_o) \cup (B, t') 
\] 
where $S^3$ is the union of the two three-balls $B$ and $B_o$ glued along their boundary, the sphere $\partial B = \partial B_o$ intersects each of $L$ and $L'$ transversely in four points, and $(B, t)$, $(B, t')$, and $(B_o, t_o)$ are two-string tangles where $t=(B \cap L)$, $t'=(B \cap L')$ and  $t_o = (B_o\cap L) = (B_o\cap L')$. Note that the ``outside" tangle $(B_o, t_o)$ is shared by both of $(S^3, L)$ and $(S^3, L')$. In particular, we may isotope $L$ and $L'$ (possibly complicating the outside tangle) until $(B, t)$ and $(B, t')$ are the specific two-string rational tangles $(B, t)=(\hsmoothing)=(0)$ and $(B, t')=(\smoothing)=(\infty)$, where $(0)$ and $(\infty)$ correspond to the Conway notation. See \cite{Kawauchi, Murasugi:book} for a general discussion of tangles. 

The Montesinos trick \cite{Montesinos} implies that the branched double covers of knots (or links) related by a band surgery are obtained from distance one Dehn fillings of a three-manifold $M$ with torus boundary. This three-manifold $M$ is the double cover of the ball $B_o$ branched over $t_o$. 
Alternatively, $M$ can be described as $\Sigma(L)-N(K)$, where the knot $K$ is the lift in the branched cover of the properly embedded arc arising as the core of the band in $(B, t)$. The filling slope $\alpha$ yielding $\Sigma(L)$ is the meridian of $K$, and the filling slope $\beta$ yielding $\Sigma(L')$ is distance one from $\alpha$, meaning $\alpha$ and $\beta$ intersect geometrically once. 

\subsection{Signature-based obstruction to the existence of a band surgery relating two knots}
\label{sec:cor}
Here we give Theorem \ref{thm:sigdif} as a corollary to Theorem \ref{thm:main}. 
Recall that the double cover $\Sigma(K)$ of the three-sphere branched over a knot $K$ is a rational homology sphere with $H_1(\Sigma(K))$ of odd order. 
\os~\cite{OS:Absolutely} defined an integer-valued knot invariant
\begin{equation}
\label{eqn:delta}
	\delta(K) = 2d( \Sigma(K), \mft_0),
\end{equation}
where $\mft_0$ is the $\spinc$ structure induced by the unique spin structure on $\Sigma(K)$. For alternating knots, and for knots of small crossing number, Manolescu and Owens proved that $\delta(K)$ is determined by the signature $\sigma(K)$ \cite{MO:Concordance}.
\begin{theorem}[Theorems 1.2 and 1.3 in \cite{MO:Concordance}]
\label{thm:delta}
If the knot $K$ is alternating or if $K$ has nine or fewer crossings, then $\delta(K) = -\sigma(K)/2$.
\end{theorem}

In \cite[Theorem 1]{LO} Lisca and Owens generalized Theorem \ref{thm:delta} to quasi-alternating knots. The class $QA$ of quasi-alternating links generalizes alternating links \cite{OS:Branched} (see Definition \ref{def:qa}). Tables of quasi-alternating knots up to 12 crossings can be found in \cite{Jablan}.

\begin{figure}
\centering
\includegraphics[width=0.2\textwidth]{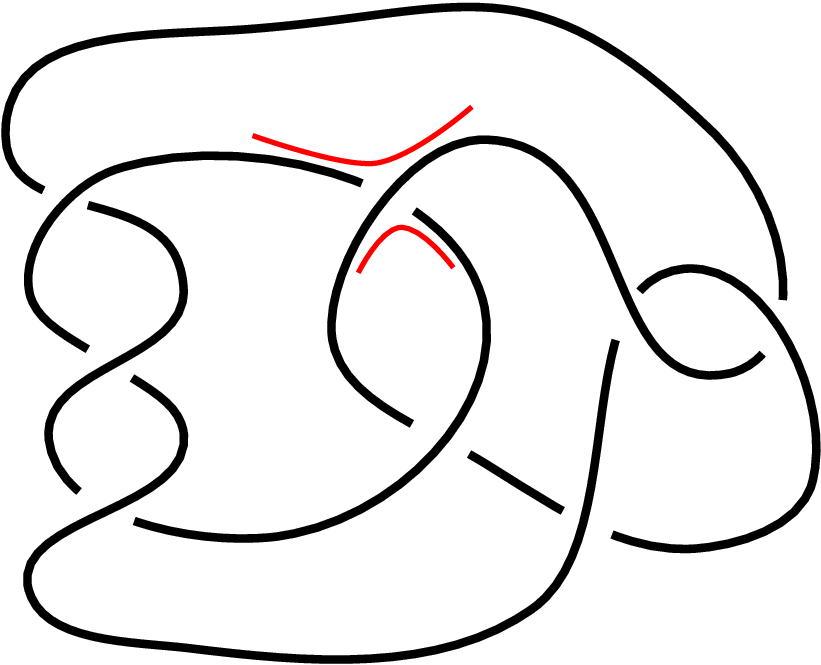}
\caption{A non-coherent band surgery taking $8_{19}$ to $3_1^*$ is obtained by the unoriented skein resolution indicated with the red arcs.}
\label{fig:819to31}
\end{figure}

\begin{definition}
\label{def:qa}
The set $QA$ of quasi-alternating links is the smallest set of links satisfying that the unknot is in $QA$, and that the set is closed under the following operation. If $L$ is a link which admits a diagram with a crossing such that 
		\begin{enumerate}
			\item both resolutions $L_0=(\hsmoothing)$ and $L_\infty=(\smoothing)$ are in $QA$
			\item $\det(L_0) \neq 0$, $\det(L_\infty)\neq 0$, and
			\item $\det(L) = \det(L_0) + \det(L_\infty)$,
		\end{enumerate}
	then $L$ is in $QA$. 
\end{definition}

In sum, the signature of a knot is determined by $\delta(K)= 2d( \Sigma(K), \mft_0)$ for alternating knots, quasi-alternating knots, and all knots of up to nine-crossings.

\begin{corollary}
\label{cor:sigdif}
Let $K$ and $K'$ be a pair of quasi-alternating knots and suppose that $H_1(\Sigma(K)) \cong \Z/d_1 \oplus \Z/d_2 \oplus \cdots \oplus \Z/d_k \cong H_1(\Sigma(K'))$, where each $d_i$ is a square-free integer. If there exists a band surgery relating $K$ and $K'$, then $|\sigma(K) - \sigma(K')|$ is $0$ or $8$. 
\end{corollary}

\begin{proof}
If $K$ and $K'$ are quasi-alternating knots, then their branched double covers $\Sigma(K)$ and $\Sigma(K')$ are L-spaces \cite{OS:Branched} and by assumption, $H_1(\Sigma(K)) \cong \Z/d_1 \oplus \Z/d_2 \oplus \cdots \oplus \Z/d_k \cong H_1(\Sigma(K'))$.  If $K$ and $K'$ differ by a band move, then $\Sigma(K')$ is obtained by a distance one filling along a knot in $\Sigma(K)$ (and vice versa). Therefore,
\begin{eqnarray*}
	& | d(\Sigma(K), \mft_0) - d(\Sigma(K', \mft_0') | &= 2 \text{ or } 0 \\
	  \Rightarrow& |\delta(K) - \delta(K') | &= 4	\text{ or } 0  \\
	  \Rightarrow& |\sigma(K) - \sigma(K') | &= 8	\text{ or } 0  
\end{eqnarray*}
where the first line follows from Theorem \ref{thm:main}, the second is the definition of $\delta$ in equation \eqref{eqn:delta}, and the third line follows from Theorem \ref{thm:delta} and its generalization \cite[Theorem 1]{LO}.
\end{proof} 

Note that Theorem \ref{thm:sigdif} as stated in the introduction is the specific case of Corollary \ref{cor:sigdif} when each $d_i$ is a distinct prime. 

\begin{theorem}
\label{thm:sigdif}
Let $K$ and $K'$ be a pair of quasi-alternating knots and suppose that $\det(K) = m = \det(K')$ for some square-free integer $m$. If there exists a band surgery relating $K$ and $K'$, then $|\sigma(K) - \sigma(K')|$ is $0$ or $8$. 
\end{theorem}

\begin{corollary}
\label{cor:lowcrossing}
Excluding $8_{19}$, let $K$ and $K'$ be knots of eight or fewer crossings with $\det(K) = m = \det(K')$ for $m$ a square-free integer. If there exists a banding from $K$ to $K'$, then $|\sigma(K) - \sigma(K')| = 0$ or $8$.
\end{corollary}
\begin{proof}
All knots of up to eight crosings are alternating with the exception of $8_{19}$, $8_{20}$ and $8_{21}$. Of these, $8_{20}$ and $8_{21}$ are quasi-alternating. So excluding $8_{19}$, all knots of up to eight crossings are quasi-alternating. The statement follows from Theorem \ref{thm:delta}. 
\end{proof}

\begin{table}
\caption{Nomenclature conversion chart. Here $K^*$ denotes the mirror of $K$. Our convention on nomenclature for mirroring agrees with that of \cite{Witte}, and the slogan that ``positive knots have negative signature" \cite{Traczyk, CochranGompf}. In the last line $n$ can be any odd integer.}
\centering
  \begin{tabular}{| r | r | r | r |} 
    \hline
	Writhe-guided & Rolfsen & Knotplot &  \\
	\hline
	$3_1$ & $3_1^*$ & $3_1$ & $T(2, 3)$ \\
	$5_1$ & $5_1^*$ & $5_1$ & $T(2, 5)$ \\
	$5_2$ & $5_2^*$ & $5_2$ &\\
	$6_2$ & $6_2^*$ & $6_2$ &\\
	$7_1$ & $7_1^*$ & $7_1$ & $T(2, 7)$ \\	
	$8_8$ & $8_8$ & $8_8^*$ &\\
	$8_{19}$ & $8_{19}$ & $8_{19}^*$ & \\
	$8_{20}$ & $8_{20}^*$ & $8_{20}$ &\\
	$n_1$ & $n_1^*$ & $n_1$ & $T(2, n)$ \\	
	\hline 	
  \end{tabular}

\label{table:nomenclature}
\end{table}

\begin{example}
Recall that the branched double cover of any alternating or quasi-alternating knot is an L-space. The knot $8_{19}$ is not quasi-alternating (in fact, it is homology-thick) and its branched double cover is not an L-space, hence it is excluded from Corollary \ref{cor:lowcrossing}. 
However, a band surgery relates $3_1^*$ and $8_{19}$. (See Figure \ref{fig:819to31}.) Here, $\det(3_1^*) = 3 = \det(8_{19})$ is square-free, and $|\sigma(3_1^*) - \sigma(8_{19})|=|2-(-6)| =8$. It is unknown whether the criterion of Corollary \ref{cor:lowcrossing} always holds for $8_{19}$.
\end{example}

\begin{example}
\label{example:knots}
	Consider the knots $6_2$ and $7_2$, for which $\det(6_2)= 11=\det(7_2)$ and $\sigma(6_2)= -2 = \sigma(7_2)$.   A single band surgery relates $6_2$ and $7_2$ (and hence $6_2^*$ and $7_2^*$) \cite{Zekovic}. Theorem \ref{thm:sigdif} implies there is no band surgery relating $6_2$ and $7_2^*$. Similarly, consider $7_1$ and $5_2$. In this case, $\det(7_1) = \det(5_2) = 7$, while $\sigma(7_1) = -6$ and $\sigma(5_2) = -2$, so there cannot be a band move relating $7_1$ and $5_2$. However, Figure \ref{fig:567} illustrates the existence of a band surgery relating $7_1$ and $5_2^*$ \cite{Zekovic}. For the knots in this article, Table \ref{table:nomenclature} provides a nomenclature conversion to Rolfsen \cite{Rolfsen}.
\end{example}

\subsection{Chirally cosmetic bandings}
\label{sec:chiral}
A \emph{chirally cosmetic} banding refers to a non-coherent band surgery taking a knot to its mirror image. The results in the previous section immediately imply the following corollary.

\begin{corollary}
\label{cor:torus}
The only nontrivial torus knot $T(2, m)$, with $m$ square free, admitting a chirally cosmetic banding is $T(2, 5)$.
\end{corollary}

\begin{proof} The statement follows from Corollary \ref{cor:lens} because the branched double cover of $T(2, m)$ is the lens space $L(m, 1)$. Alternatively, we may use Theorem \ref{thm:sigdif}. Without loss of generality, let $m>0$. The torus knot $T(2, m)$ has determinant $m$ and signature $1-m$, whereas the mirror $T(2, -m)$ has determinant $m$ and signature $-1+m$. The absolute value of the difference in signature between $T(2, m)$ and $T(2, -m)$ is $0$ when $m=1$ and is $8$ when $m=5$. The case of $m=1$ is the unknot. The chirally cosmetic banding for the case $m=5$ is pictured in Figure \ref{fig:chiralexamples}. 
\end{proof}

\begin{figure}
\centering
\subfloat[The knots $8_8$ and $8_{20}$.]{\label{fig:symmetricunions} \includegraphics[width=5.2cm]{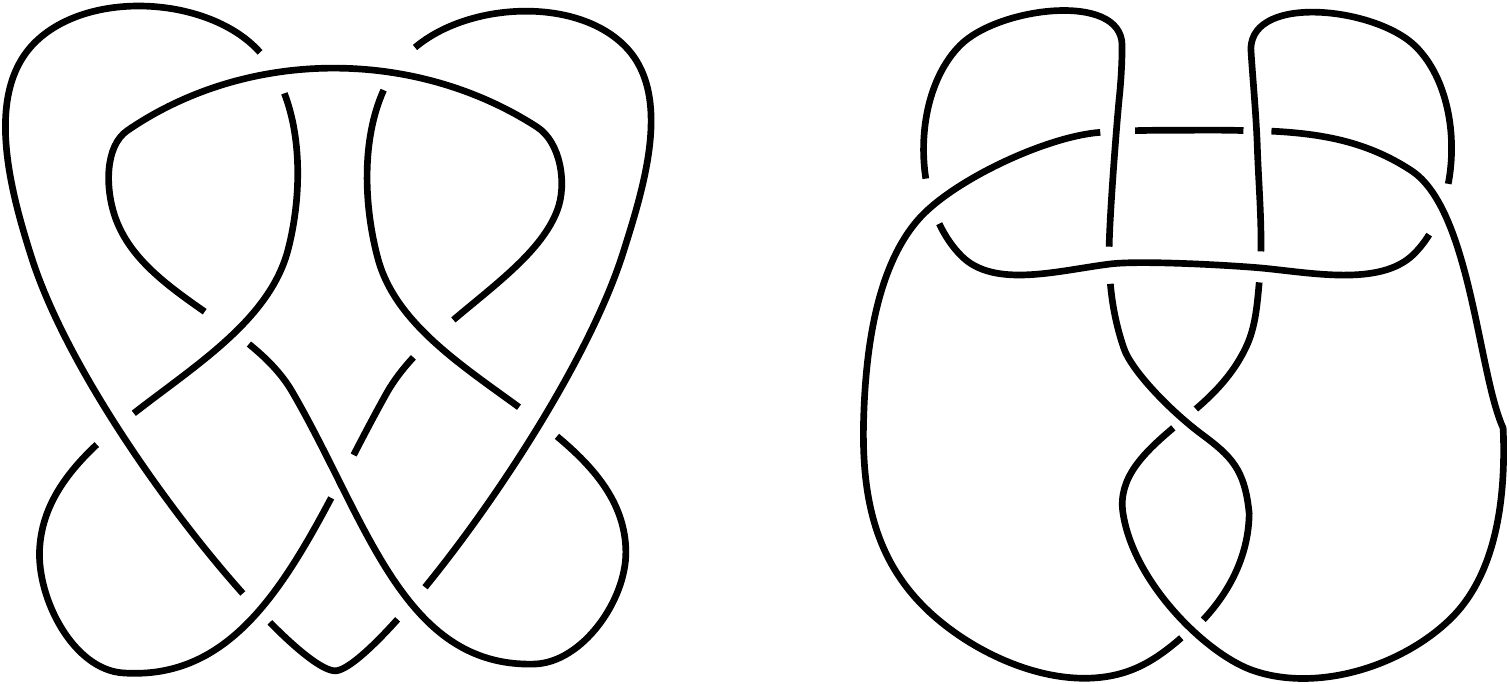}}
\qquad
\subfloat[Chirally cosmetic banding along $8_8^*$ and $8_{20}$.]{\label{fig:chiralbanding} \includegraphics[width=1.4in]{8_8swithband-long.pdf} \qquad \includegraphics[width=0.8in]{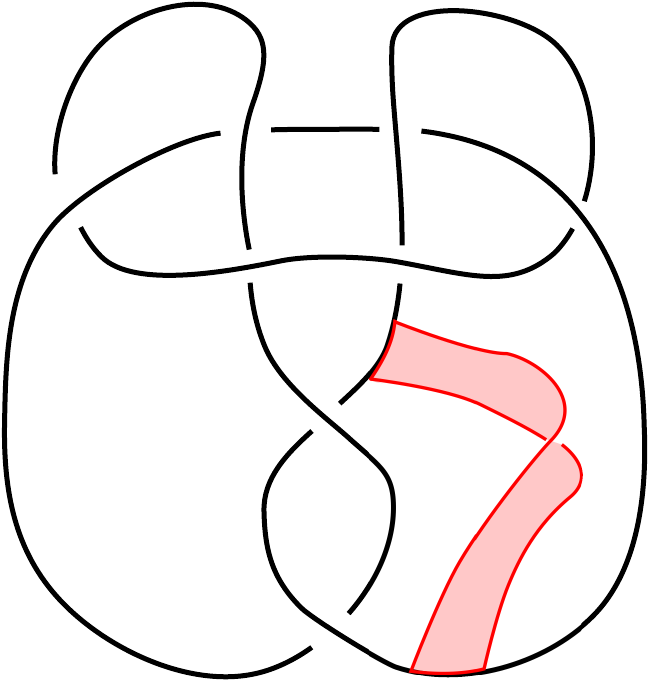}}%
\caption{(A) The knots $8_8$ and $8_{20}$ are symmetric unions. 
(B) Chirally cosmetic bandings relating the pair $8_8^*$ and $8_8 $ (left) and the pair $8_{20}$ and $8^*_{20}$ (right). The banding for $8_8^*$ was discovered via computer simulation. The banding for $8_{20}$ induces a ``4-move."}
\label{fig:symmetricunionsandbands}
\end{figure}

The knot $K$ in $L(5, 1)$ along which there exists a distance one filling yielding $-L(5, 1)$ descends under the covering involution to the core of the band move from $T(2, 5)=5_1$ to $T(2, -5)=5_1^*$. As observed in \cite{IJ}, the complement of this knot is the hyperbolic knot complement known as the ``figure-eight
sibling" and is well known to be amphichieral \cite{MP, Weeks}. The special symmetries of this manifold suggest an explanation for the uniqueness of $T(2, 5)$; see \cite{IJ} for a discussion of this perspective. 

In addition to $5_1$, several constructions of knots admitting chirally cosmetic bandings are described in \cite{IJ}. These include: the knot $9_{27}$, which has determinant 49; Whitehead doubles of achiral knots (i.e. certain satellites), which have determinant 1; and a general construction of certain symmetric unions. A symmetric union is a connected sum of a knot and its mirror image, modified by a tangle replacement such that the resulting diagram admits an axis of mirror symmetry (see Figure \ref{fig:symmetricunions}). The determinant of a symmetric union is always a square \cite{KT, Moore}. Indeed, excluding $5_1$, all of the examples of knots admitting chirally cosmetic bandings presented in \cite{IJ} have determinant a square.

\textbf{Numerical exploration of chirally cosmetic bandings.} Corollary \ref{cor:torus} and the observation that most known constructions of knots admitting chirally cosmetic bandings are symmetric unions suggest that this phenomenon is uncommon. 
To search for any further chirally cosmetic bandings amongst low crossing knots, we performed computer simulations (described at the end of this section) of non-coherent band surgery along knots. 
In brief, each knot type $K$ was embedded in the cubic lattice. 
The embeddings were randomized using a Monte Carlo algorithm, where 
$6\times 10^6$ polygons were sampled for each chiral pair and a non-coherent band surgery was performed on each polygon.

\begin{table}[t]

\caption{
\textbf{Chirally cosmetic bandings observed in numerical experiments.} 
 Here we report the number of chirally cosmetic bandings for knots of up to eight crossings observed by computer simulation, where $6\times 10^6$ total band moves were performed on each chiral pair. Because the BFACF algorithm may change the length of the polygon, the last column indicates the length range observed in sampled polygons.
}
\begin{center}
  \begin{tabular}{ r  r  c  }
    \hline
    \emph{Chiral pairs} & 
    \emph{Number observed} & [$\ell_{min}, \ell_{max}$]  \\
    \hline
    $(5_1, 5_1^*)$ & 188 & $[42, 2016]$ \\
    $(8_8, 8_8^*)$ & 1 & $[100, 2894]$ \\
    $(8_{20}, 8_{20}^*)$ & 2677 & $[68, 2370]$ \\    
    \hline
  \end{tabular}
\end{center}

\label{table:chiral}
\end{table}

We found chirally cosmetic bandings to be exceedingly rare. 
Chirally cosmetic bandings were observed in the numerical experiments for only three knot types. 
The numbers of bandings relating each of these three knots with its mirror image (or vice versa) is given in Table \ref{table:chiral}.

The band move for $5_1$ was previously reported in \cite{Zekovic} and is shown in Figure \ref{fig:chiralexamples}. 
Only a single band move was observed for the pair $8_8$ and $8_8^*$. 
An isotopy of this banding, discovered only by simulation, is shown in Figure \ref{fig:chiralbanding} (left).   
A banding for $8_{20}$ can more easily be found by hand, as in Figure \ref{fig:chiralbanding} (right).  
The banding pictured is a ``4-move" in a symmetric union diagram, where a 4-move refers to the replacement of a positive 2-twist with a negative 2-twist. 
Although $8_8$ is known to admit a symmetric union presentation \cite{LAMM} (see Figure \ref{fig:symmetricunions}), it is not known to the authors whether a 4-move in such a diagram relates $8_8$ to its mirror. 
See \cite{IJ} for a definition of 4-moves.

These observations prompt the following question:
\begin{question}
\label{question}
Suppose that a pair of knots $K$ and $K'$ are related by band surgery. 
What is the likelihood that $K$ and $K'$ are of different chirality?
\end{question}
When $K$ is a chiral knot and $K'$ is the mirror of $K$, it is clear what we mean by \emph{different chirality}. 
In a forthcoming numerical study propose a method to quantify what it means for arbitrary knots to have different chirality. We expand on the preliminary simulations conducted here and report on the chirality trends with respect to band surgery operations along cubic lattice knots.

\textbf{Numerical methods.} In order to identify candidate pairs $(K, K^*)$ for chirally cosmetic bandings, we systematically performed band surgery operations on all nontrivial prime chiral knots with 8 or fewer crossings. Independent simulations were performed for each member of a chiral pair ($K, K^*)$. 
Because we excluded fully amphichiral and negative amphichiral knots, a total of 28 knot types were considered.
The computer simulations were conducted by adapting software previously developed in our group, the implementation and methods of which are detailed in \cite{Stolz}.

The details of the simulations are as follows. 
For each knot $K$, we used the BFACF algorithm \cite{MadrasSlade} to sample the space of lattice embeddings of $K$. 
Here a cubic lattice representative of a knot $K$ is an embedding of $K$ contained in the cubic lattice $(\mathbb{R}\times \Z\times \Z) \cup (\Z\times\mathbb{R}\times \Z)\cup (\Z\times\Z\times\mathbb{R})$, where the end points of the line segments are contained in the integer lattice $\Z^3$. 
The BFACF algorithm is a dynamic Monte Carlo method that samples the space of cubic lattice polygons of a specific knot type. 
In each step of the BFACF Markov chain, the algorithm attempts to locally perturb a polygonal chain as pictured in Figure \ref{fig:bfacfmoves}, at a randomly chosen location on edges in the polygon. 
Note that this can change the length of the polygon by $-2, 0$, or $2$. 
Janse van Rensburg and Whittington showed that the ergodicity classes are the knot types \cite{RensburgWhittington}, making this a good method for generating large ensembles of lattice embeddings of a specific knot type. 
The algorithm depends on one adjustable parameter $z$. 
For any given starting knot type $K$, fixing $z$ yields uniform sampling of polygons of type $K$ and fixed mean length.

\begin{figure}
	\includegraphics[width=5in]{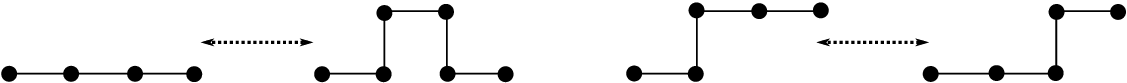}
	\caption{The elementary moves of the BFACF algorithm.}
	\label{fig:bfacfmoves}
\end{figure}

Our implementation is based on the algorithm described in \cite{MadrasSlade} and runs in constant time on the number of edges in the polygons. 
In particular, we use the BFACF algorithm in conjunction with a multiple Markov chain procedure to generate a large, random ensemble of polygonal representatives of each fixed knot type $K$. 
From this much larger ensemble, we uniformly sample a smaller subset of $3\times 10^6$ polygons. 
Amongst the sample set, we search for candidate banding locations by finding parallel unit edges at distance one, pick one such site at random and perform one non-coherent band surgery operation. 
Locally, the operation is as illustrated in Figure \ref{fig:567}. 
We identify the knot type of the resulting knot using an implementation of the HOMFLY-PT polynomial \cite{F, PT} based on published algorithms \cite{G, J}. 
In cases where the HOMFLY-PT may lead to ambiguous knot identification, we apply KnotFinder \cite{indiana}. 
The outcome is a directed, weighted graph where the nodes are the knot types and the edges indicate transitions between knot types with corresponding transition probabilities. 

It is worth remarking that the sample sets of conformations generated for a knot $K$ and for its mirror image $K^*$ are not related by mirroring in the cubic lattice, as each set is a uniform random sample of a large set of polygons generated to represent each knot. 
For example, a non-coherent banding was performed for each of $3\times 10^6$ lattice embeddings of $8_8$ and $3\times 10^6$ lattice embeddings of $8_8^*$. 
Out of these bandings, only one was a chirally cosmetic banding relating $8_8^*$ to $8_8$. 
Here, because the experiments were designed to find rare events, the lengths of the polygons was allowed to vary widely, as can be seen in Table \ref{table:chiral}. 
The BFACF algorithm guarantees uniform random sampling for each fixed knot type and fixed length. 
If the length is fixed in the simulations, then when the number of iterations is sufficiently large we expect the behavior of a knot to be equal to that of its mirror image, within a small margin of error. 
Of course, every time a banding is identified taking $K$ to $K^*$, after mirroring there exists a cosmetic banding relating $K^*$ to $K$. 
However, for the reasons stated above, the number of transitions observed in simulation for a knot and its mirror are not identical.

\subsection*{Acknowledgements}
Many of the ideas in the article originated in discussions with Tye Lidman, who we thank for his willingness to discuss the mapping cone construction. We thank Cameron Gordon for helpful discussions and Michelle Flanner for preliminary computational data. 
We are grateful to In Dae Jong and Kazuhiro Ichihara for alerting us to an early error in knot identification pertaining to the skein equivalent knots $8_8$ and $10_{129}$. AM and MV were partially supported by DMS-1716987. MV was also partially supported by DMS-1817156.

\bibliographystyle{alpha}
\bibliography{bibliography}

\end{document}